\documentclass[12pt, dvipsnames]{amsart}

\usepackage[english]{babel}

\usepackage[T1]{fontenc}

\usepackage{Alegreya}      
\usepackage{AlegreyaSans}  

\usepackage{tikz}
\usepackage{tikz-cd}
\definecolor{my1}{RGB}{30,144,255}
\definecolor{my2}{RGB}{255,20,147}
\usepackage[colorlinks=true,linkcolor=my2, citecolor=my1]{hyperref}
\usepackage{amsmath, amsfonts, mathrsfs, amsthm, amstext}
\usepackage{enumerate}
\usepackage{array}
\usepackage{enumitem}
\usepackage{hologo}
\usepackage{graphicx}
\usepackage{multirow}
\usepackage[colorinlistoftodos]{todonotes}

\makeatletter
\DeclareFontEncoding{LS1}{}{}
\DeclareFontSubstitution{LS1}{stix}{m}{n}
\makeatother

\DeclareMathAlphabet\mathscr{LS1}{stixscr}{m}{n}
\SetMathAlphabet\mathscr{bold}{LS1}{stixscr}{b}{n}

\usepackage[hmargin=2.3cm, vmargin=3cm]{geometry}
 
\usepackage{booktabs}
\usepackage{pifont}

\title{Brownian motion on reflection quantum groups. Construction and cutoff.}
\author{Jean DELHAYE}
\thanks{}

\renewcommand{\and}{\quad \hbox{\&}\quad}
\newcommand{\et}{\quad \hbox{\&}\quad}

\newtheorem{thm}{Theorem}[section]
\newtheorem{lem}[thm]{Lemma}
\newtheorem{prop}[thm]{Proposition}
\newtheorem{cor}[thm]{Corollary}

\theoremstyle{definition}

\newtheorem{conj}[thm]{Conjecture}

\newtheorem{de}[thm]{Definition}

\newtheorem{rem}[thm]{Remark}

\DeclareMathOperator{\Meix}{Meix}

\DeclareMathOperator{\Irr}{Irr}

\DeclareMathOperator{\id}{id}

\newcommand{\sg}{>}
\renewcommand{\sl}{<}
\newcommand{\vphi}{\varphi}

\newcommand{\C}{\mathbb{C}}

\newcommand{\E}{\mathbb{E}}
\newcommand{\p}{\mathbb{P}}
\newcommand{\F}{\mathbb{F}}
\newcommand{\G}{\mathbb{G}}

\newcommand{\N}{\mathbb{N}}
\newcommand{\R}{\mathbb{R}}
\newcommand{\Z}{\mathbb{Z}}

\newcommand{\longto}{\longrightarrow}

\renewcommand{\d}{\mathrm{d}}

\renewcommand{\O}{\mathcal{O}}

\counterwithin{equation}{section}

\allowdisplaybreaks

\usepackage{fancyhdr}
\begin{document}

\begin{abstract}
    In this study, we construct an analog of the Brownian motion on free reflection quantum groups and compute its cutoff profile.
\end{abstract}

\maketitle

{\hypersetup{hidelinks}
\tableofcontents}

\section{Introduction}

Given a sequence of compact groups $(G_N)_{N \in \N}$, each equipped with a Markov process $(X_t^{(N)})_{t \geq 0}$, we say that it exhibits a \emph{cutoff} at time $(t_N)$ if the total variation distance to the Haar probability measure $m_N$ drops within a window of order $o(t_N)$, that is if 
$$
d_N\big(t_N(1-\epsilon)\big)\underset{N\to\infty}{\longto} 1
\quad \text{and} \quad
d_N\big(t_N(1+\epsilon)\big)\underset{N\to\infty}{\longto} 0,
\quad \epsilon \sg 0,
$$
where 
$$
d_N(t) = \sup_{A\subset G_N}\left| \p[X_t^{(N)} \in A] - m_N(A) \right|
$$
denotes the total variation distance between the process at time $t$ and the Haar probability measure.  
More precisely, if there exists a continuous function $f:\R \to [0,1]$, non-increasing from $1$ to $0$, such that 
$$
d_N\big(t_N + cs_N\big) \underset{N\to\infty}{\longto} f(c),\quad c\in \R,
$$
then $f$ is called a \emph{cutoff profile} of the process.

The first instance of the cutoff phenomenon appeared in the work of P. Diaconis and M. Shahshahani on random transpositions \cite{DS81}. Many other results followed, both in algebraic settings and beyond, as those in \cite{Ald83, BD92, BLPS18, LLP10, LS10, Mel14}.
It was later extended by J. P. McCarthy to the quantum setting in \cite{Mcc17, Mcc19}. Once again, other examples followed in this framework, for instance in \cite{Fre18, Fre19}.
Finally, we point out several works in which not only is the cutoff illustrated, but cutoff profiles are also computed (in both classical and quantum settings), such as in \cite{BN22, FTW21, Lac16, NT22, OTT25, DGM90, Tey19}.\\

The goal of this paper is twofold. First, we aim to define a Brownian motion on the quantum reflection groups $H_N^{s+}$. While Brownian motions have already been constructed on certain compact quantum groups, these constructions remain highly case-specific. On $O_N^+$ and $S_N^+$, the process arises through analogues of the Hunt-type formula for central generating functionals, whereas on $U_N^+$ it is obtained via the centralization of Gaussian generating functionals. However, no general framework exists for defining Brownian motions on arbitrary compact quantum groups, and in particular, none of these methods apply to $H_N^{s+}$. In Section \ref{construction}, we propose a natural construction of such a process, consistent with the known quantum cases. Second, in Section \ref{profile}, we establish the existence of a cutoff profile for this process and prove a genuine dependency in the parameter $s\in \N^*$.

\begin{thm}
    The Brownian motion on $H_N^{s+}$ has cutoff at time $t_N = N\ln N$. More precisely, it exhibits the following cutoff profile:
    $$
    d_N\left(
    N\ln \left(
    \sqrt{s}N
    \right)+cN
    \right) \underset{N\to\infty}{\longto}
    d_{\mathrm{TV}}
    \left(
    \eta_c^s,\nu_{s}
    \right)
    ,\quad c \sg 0,
    $$
    and 
    $$
    \underset{N\to\infty}{\lim\inf}\, d_N\left(
    N\ln \left(
    \sqrt{s}N
    \right)+cN
    \right) \ge  d_{\mathrm{TV}} (\eta_c^s,\nu_s),\quad c\in \R,
    $$
    where $\nu_{s}$ and $\eta_c^s$ are the following probability measures: 
    $$
\nu_s = \frac{\sqrt{4-(x-\sqrt s)^2}}{2\pi (x\sqrt s+1)}
\mathbf 1_{\vert 
x-\sqrt s
\vert \le 2}
\d x + \left(
1-s^{-1}
\right) \delta_{-1/\sqrt s}= \mathrm{Pois}^+(s^{-1},\sqrt s)*\delta_{-1/\sqrt s} = \mathrm{Meix}^+(\sqrt s,1),
$$
    and 
    $$
    \eta_c^s = f_c\d \nu_s + 
    \mathbf 1_{c\sl 0}
    \frac{e^{-c}-e^c}{e^{-c}+\sqrt s}
    \delta_{e^c+\sqrt s+e^{-c}},
    \quad \hbox{with }
    f_c(x) = \frac{e^c+\sqrt s}{e^c+e^{-c}+\sqrt s-x}.
    $$
\end{thm}

\section{Preliminaries}\label{preliminaries}

In this section, we introduce Lévy processes and the cutoff phenomenon on compact matrix quantum groups in the same spirit as in \cite{FTW21} and \cite{Del24} and will therefore produce a similar introduction, focusing on reflection quantum groups.

\subsection{Free reflection quantum groups}

Free reflection quantum groups, a specific class of compact quantum groups, were introduced by S. Wang in \cite{Wan95}. Quantum groups, more broadly, were originally defined by S.L. Woronowicz \cite{Wor98}. These structures involve C*-algebras, fitting their noncommutative topological nature. However, in this article, we will present an alternative definition focusing on the algebraic aspects, making it more accessible to non-expert readers.

Let us recall that a $*$\emph{-algebra} is a complex unital algebra $A$ endowed with an \emph{involution} $A\to A,x\mapsto x^*$, i.e. an antimultiplicative antilinear map such that $x^{**} = x$ for all $x\in A$. Also, a $*$-ideal $B$ of $A$ is a $*$-subalgebra such that $aB\cup Ba\subset B$ for all $a\in A$.

\begin{de}\label{def of H_Ns+}
We define $\O(H_N^{s+})$ to be the universal $*$-algebra generated by $N^2$ normal elements $\{u_{ij}\}_{1\le i,j\le N}$ such that 
\begin{enumerate}[label = (\roman*)]
    \item $u = (u_{ij})$ and $u^t = (u_{ji})$ are unitary;
    \item $p_{ij} := u_{ij}u_{ij}^*$ is a projection;
    \item $u_{ij}^s = p_{ij}$.
\end{enumerate}
\end{de}

Let $H_N^s$ be the classical reflection group, i.e. the group of $N\times N$ complex valued matrices that have exactly one nonzero coefficient on each row and column that is an $s$-th root of the unity. Let $\O(H_N^s)$ be the algebra of \emph{regular functions} on $H_N^s$, i.e. the $*$-algebra generated by the functions $c_{ij}:H_N^s\to \C$ sending a matrix to its $(i,j)$ coefficient, where the involution corresponds to the complex conjugation: $c_{ij}^* = \overline c_{ij}$. Then, quotienting $\O(H_N^{s+})$ by its commutator ideal yields a surjection 
$$
\pi: \left\{
\begin{matrix}
\O(H_N^{s+})&
\longto&
\O(H_N^s)\\
u_{ij}&
\longmapsto&
c_{ij}.
\end{matrix}
\right.
$$
In that sense, $H_N^{s+}$ can be thought of as a noncommutative version of $H_N^s$. The group structure of $H_N^s$ can be recovered in an algebraic way on $\O(H_N^s)$ via the following remark:
$$
c_{ij}(gh) = \sum_{k=1}^N c_{ik}(g)c_{kj}(h) = \sum_{k=1}^N (c_{ik}\otimes c_{kj})(g,h),\quad g,h\in H^s_N, 1\le i,j\le N.
$$
The ``group law" on $H_N^{s+}$ is therefore given by the unique $*$-homomorphism $\Delta:\O(H_N^{s+})\to \O(H_N^{s+})\otimes \O(H_N^{s+})$, called \emph{coproduct}, such that 
$$
\Delta(u_{ij}) = \sum_{k=1}^N u_{ik}\otimes u_{kj}.
$$
The existence of $\Delta$ follows from the universal property. Additionally, let us outline the existence of the \emph{counit} and the \emph{antipode}, which are crucial objects. The counit is given by the homomorphism $\varepsilon: \O(H_N^{s+}) \to \C$, defined by $\varepsilon(u_{ij}) = \delta_{ij}$. The antipode, on the other hand, is the antihomomorphism $S: \O(H_N^{s+}) \to \O(H_N^{s+})$, defined by $S(u_{ij}) = u_{ji}^*$. These maps serve as the unit and the inverse map in the noncommutative setting, respectively.

In this context, probability measures can be generalized by identifying them with their integration linear forms. These forms correspond to \emph{states}, which are unital positive linear forms on $\O(H_N^{s+})$. Notably, there exists a particular state that serves as the analog of the Haar measure on $H_N^{s+}$ (see \cite{Wor98}).

\begin{thm}
There is a unique state $h$ on $H_N^{s+}$ such that 
$$
(\id \otimes h) \circ \Delta(x) = h(x) \otimes 1 = (h \otimes \id) \circ \Delta(x),\quad x\in \O(H_N^{s+}).
$$
It is called the \emph{Haar state} of $H_N^{s+}$. 
\end{thm}

\begin{rem} 
Let us point out the fact that the Haar state is \emph{tracial}, i.e. we have $h(xy) = h(yx)$ for any $x,y \in \O(H_N^{s+})$.
\end{rem}

The representation theory is a great tool to study the asymptotic behaviour of random walks on compact groups (see, for instance, \cite[Chap 4]{Dia88}). Let us describe the irreducible representations of the reflection quantum group (see \cite[Thm 8.2]{BV09}).

\begin{thm}\label{irr characters}
    The irreducible characters of $H_N^{s+}$ may be labelled by the elements of the submonoid $M = \langle az^ra: r\in \{0,\cdots,s-1\} \rangle \subset  \langle a,z : z^s=1\rangle$ equipped with the involution $a^* = a$ and $z^* = z^{-1}$ with 
    $$
    \chi_{\emptyset} = 1,\quad \chi_{a^2} = \sum_i p_{ii}-1 \emph{\and} \chi_{az^ra} = \sum_i u_{ii}^r,\quad 0\le r\le s-1,
    $$
    and the following fusion rule
    $$
    \chi_{vaz^i}\cdot \chi_{z^jaw} = 
    \left\{
    \begin{array}{ll}
    \chi_{vaz^{i+j}aw} & \hbox{if }i+j\ne0\!\!\!\!\! \mod s \\
    \chi_{va^2w}+ \chi_v\cdot \chi_w &\hbox{else}. 
\end{array}\right.
    $$
    More precisely, the elements of $M$ are of the form  
    $$
    a^{\ell_1}z^{r_1}\cdots z^{r_{k-1}}a^{\ell_k}
    $$
    with $r_i\in \Z_s^*$, $\ell_1$ and $\ell_k$ being odd and all other $\ell_i$'s being even except if $k=1$ in which case $\ell_1$ is even. 
\end{thm}

We denote by $\O(H_N^{s+})_0$ the algebra generated by the characters and call it the \emph{central algebra} of $H_N^{s+}$. An important feature of this subalgebra is the existence of a conditional expectation (see \cite[Chap 9]{AP}) $\E: \O(H_N^{s+})\to \O(H_N^{s+})_0$ that leaves the Haar state invariant, let us recall the definition of such an object.

\begin{de}
Let $A$ be a $C^*$-algebra (that is a Banach $*$-algebra such that $\Vert a^*a\Vert = \Vert a\Vert^2$ for all $a\in A$) and $B\subset A$ a subalgebra of $A$. We call \emph{conditional expectation} from $A$ onto $B$ a linear map $\E:A\to B$ satisfying:
\begin{enumerate}[label = (\roman*)]
    \item $\E[A_+]\subset B_+$ where $A_+ = \{a^*a : a\in A\}$ and $B_+ = A_+\cap B$;
    \item $\E[b] = b$ for all $b\in B$;
    \item $\E[bab'] = b\E[a]b'$ for all $a\in A$ and $b,b'\in B$.
\end{enumerate}
\end{de}

Note that a conditional expectation is always a norm-one projection. 

Let us now describe the quantum analogue of a Lévy process on $H_N^{s+}$. We call \emph{(quantum) Lévy process} on $H_N^{s+}$ any right-continuous convolution semigroup of states, i.e. a family $(\psi_t)_{t\ge0}$ of states on $H_N^{s+}$ such that 
\begin{itemize}
    \item $\psi_0 = \varepsilon$;
    \item $\psi_t\star \psi_s := (\psi_t\otimes \psi_s)\circ \Delta = \psi_{t+s}$;
    \item $\psi_t\to \psi_0$ weakly as $t\to0$.
\end{itemize}
 
We call \emph{generating functional} on $H_N^{s+}$ an hermitian functional $L:\O(H_N^{s+})\to\C$ that vanishes on $1$ and is positive on $\ker \varepsilon$. There is a one-to-one correspondence between Lévy processes and generating functionals (see \cite[Sec 1.5]{FS16}) via the formulas 
\begin{align*}
    L &= 
    \lim_{t\to 0^+} \frac{\psi_t-\varepsilon}{t}, \\
    \psi_t &= 
    \exp_\star(tL) = 
    \sum_{n\ge0} 
    \frac{(tL)^{\star n}}{n!} ,\quad t\ge0.
\end{align*}

\begin{de}
We will say that a Lévy process $(\psi_t)_{t\ge0}$ is \emph{central} if all the $\psi_t$'s are central, i.e. if $\psi_t = \psi_t \circ \E$ for all $t\ge 0$. Equivalently, a Lévy process is central if its associated generating functional $L$ is \emph{central}, that is, $L = L\circ \E$.
\end{de}

The notion of centrality will become relevant in Section \ref{construction} as we introduce the Brownian motion as a central Lévy process. Observe that for any generating functional $ L $ on $ \O(H_N^{s+}) $, the composition $ L \circ \mathbb{E} $ also defines a generating functional on $ \O(H_N^{s+}) $ which we call the \emph{centralized} generating functional. Thus, we may consider centralized generating functional and when doing so have in mind that all their information is contained within the restriction of $L$ to the central algebra $\O(H_N^{s+})_0$.

\subsection{The cutoff phenomenon}\label{cutoff phenomenon}

In this subsection, we introduce the concept of the \emph{cutoff phenomenon} within the framework of quantum groups. The cutoff phenomenon, a sharp transition in the convergence to equilibrium, is a significant topic in probability theory and has been widely studied in classical settings. Here, we focus on the reflection quantum groups and thus work in the quantum framework. We will provide necessary definitions, discuss the associated Lévy processes, and outline the framework used to investigate the \emph{limit profiles}.

One may consider the universal enveloping C*-algebra $C(H_N^{s+})$ associated with $\O(H_N^{s+})$ (see, for instance, \cite[Sec II.8.3]{Bla06}). By definition, any state on $\O(H_N^{s+})$ uniquely extends to a state on $C(H_N^{s+})$, thus yielding an element of the Fourier-Stieltjes algebra, which corresponds to the topological dual $C(H_N^{s+})^*$ of $C(H_N^{s+})$. By the Riesz representation theorem, this dual space can be viewed as a noncommutative analogue of the measure algebra, equipped with the total variation norm. We denote the norm on this dual space as $\vert \cdot \vert_{FS}$, referring to it as the Fourier-Stieltjes norm. 

Furthermore, the topological double dual $C(H_N^{s+})^{**}$ of $C(H_N^{s+})$ is recognized as the universal enveloping von Neumann algebra of $C(H_N^{s+})$, serving as a noncommutative and universal analogue of classical measure spaces. Using the theory of Haagerup's noncommutative $\mathrm L^p$-spaces, we can readily adapt the argument from \cite[Lem 2.6]{Fre19} to show that:
$$
\frac{1}{2}\|\vphi - \psi\|_{FS} = \sup_{p \in \mathcal P} \vert\vphi(p) - \psi(p)\vert,
$$
where $\mathcal P$ denotes the set of orthogonal projections in $C(H_N^{s+})^{**}$, which can be thought of as the noncommutative counterparts to indicator functions of Borel subsets. The proof details are omitted here, as they are not required for our current discussion. This generalized total variation distance will be employed in our analysis of cutoff phenomena, and we will denote this distance as 
$$
d_{\mathrm{TV}}(\vphi,\psi) = \frac{1}{2} \vert \vphi -\psi\vert_{FS}.
$$
In the classical setting, a particularly significant case arises when both $\mu$ and $\nu$ are absolutely continuous with respect to the Haar measure. We can consider a similar scenario in the quantum context. Define an inner product on $\O(H_N^{s+})$ by $\langle x, y \rangle = h(xy^*)$. Completing this inner product space yields the Hilbert space $\mathrm L^2(H_N^{s+})$, and $\O(H_N^{s+})$ embeds into $\mathrm B(\mathrm L^2(H_N^{s+}))$ via left multiplication (see \cite[Cor 1.7.5]{NT13} and the comments thereafter). The weak closure of this image is denoted by $\mathrm L^\infty(H_N^{s+})$ and forms a von Neumann algebra.

If $\vphi : \O(H_N^{s+}) \to \mathbb{C}$ is a linear map that extends to a normal bounded map on $\mathrm L^\infty(H_N^{s+})$, then $\vphi$ is an element of the Fourier algebra, which is the Banach space predual $\mathrm L^\infty(H_N^{s+})_*$ of $\mathrm L^\infty(H_N^{s+})$. In this case, $\vert\vphi\vert_{FS} = \|\vphi\|_{\mathrm L^\infty(H_N^{s+})_*}$ (see \cite[Prop 3.14]{BR17}), which further implies, by \cite[Lem 2.6]{Fre19}, that:
$$
d_{\mathrm{TV}}(\vphi,\psi) =\sup_{p \in \widetilde{\mathcal P}} \vert\vphi(p) - \psi(p)\vert,
$$
where $\widetilde{\mathcal{P}}$ is the set of orthogonal projections in $\mathrm L^\infty(H_N^{s+})$. It is important to note that for this formula to be applicable, the states $\vphi$ and $\psi$ must extend to the von Neumann algebra $\mathrm L^\infty(H_N^{s+})$.

\begin{de}\label{def-cutoff}
    Let $(\G_N,(\vphi^{(N)}_t)_{t\ge0})_{N\in\N}$ be a sequence of compact quantum groups each equipped with a Lévy process. We say that they exhibit a \emph{cutoff phenomenon} at time $(t_N)_{N\in \N}$ if for any $\epsilon>0$ we have 
    $$
    d_{\mathrm{TV}}( \vphi_{t_N(1-\epsilon)}^{(N)},h_N) \underset{N\to\infty}{\longto} 1
    \et 
    d_{\mathrm{TV}}( \vphi_{t_N(1+\epsilon)}^{(N)},h_N) \underset{N\to\infty}{\longto} 0.
    $$
    where $h_N$ denotes the Haar state of $\G_N$. More precisely, given a continuous function $f$ decreasing from $1$ to $0$ such that 
    $$
     d_{\mathrm{TV}}(\vphi_{t_N + cs_N},h) \underset{N\to\infty}{\longto} f(c),\quad c\in \R,
    $$
    for some sequence $(s_N)_{N\in \N}$ with $s_N = o(t_N)$, we say that $f$ is the \emph{limit profile} of the process.
\end{de}

Note that the profile is unique up to affine transformation. Two profiles $f$ and $f'$ are equivalent if we can achieve both with the same process (through a scaling change on the sequences $(t_N,s_N)_N$). It follows from \cite[Rem 2.8]{Del24} that this only happens if $f'$ is an affine transformation of $f$, i.e.
$$
f' = f(a\, \cdot  + b),\quad \hbox{for some } a\sg 0 \hbox{ and }b\in \R.
$$

In this paper, we explore the limit profile of the Brownian motion on $H_N^{s+}$ which takes the form 
$$
c\longmapsto d_{\mathrm{TV}}(\eta_c^s,\nu_s)
$$
where $\nu_s$ is the orthonormal measure for the family of polynomials $(Q_n^s)_{n\in\N}$ defined by 
$$
Q_0^s =1,\quad Q_1^s = X \and XQ_n^s = Q_{n+1}^s +\sqrt sQ_n^s+Q_{n-1}^s,\quad n\ge 1.
$$
and $\eta_c^s$ is the unique probability measure satisfying $\eta_c^s(Q_n^s) = e^{-cn}$. 

Note that we can have a clearer decomposition as follows (see \cite[Thm 2.1]{SY01})
$$
\nu_s = \frac{\sqrt{4-(x-\sqrt s)^2}}{2\pi (x\sqrt s+1)}
\mathbf 1_{\vert 
x-\sqrt s
\vert \le 2}
\d x + \left(
1-s^{-1}
\right) \delta_{-1/\sqrt s}
$$
Note that $\nu_s$ can be seen as a shifted free Poisson law or as a Free Meixner law as follows 
$$
\nu_s = \mathrm{Pois}^+(s^{-1},\sqrt s)*\delta_{-1/\sqrt s} = \mathrm{Meix}^+(\sqrt s,1).
$$
For more information on free Poisson and free Meixner laws, we refer the reader to \cite[Def 12.12]{NS06}) and \cite[Sec 2.2]{BB06} respectively. 

\begin{lem}\label{decompo of eta_c^s}
    We have 
    $$
    \eta_c^s = f_c\d \nu_s + 
    \mathbf 1_{c\sl 0}
    \frac{e^{-c}-e^c}{e^{-c}+\sqrt s}
    \delta_{e^c+\sqrt s+e^{-c}}
    \quad \hbox{where }
    f_c(x) = \frac{e^c+\sqrt s}{e^c+e^{-c}+\sqrt s-x}.
    $$
\end{lem}

\begin{proof}
    The described measure has compact support and is therefore moment-determined so the uniqueness does follow from moment determination. 

    Let us first treat the case $c\sg 0$. It is clear that in this case, $\eta_c^s$ is absolutely continuous w.r.t. $\nu_s$ since 
    $$
    \Vert \eta_c^s\Vert_{2,\nu_s}^2 \le \sum_{n }
    \vert 
    \eta_c^s(Q_n^s)
    \vert^2 = \sum_n e^{-2cn} \sl \infty.
    $$
    Its density is then simply given by $\sum_{n\ge 0}e^{-cn}Q_n^s$ which can be shown to equal $f_c$ using the recurrence relation. The case $c=0$ follows from dominated convergence 
    \begin{align*}
    \eta_0^s (Q_n^s) &= 1
    = \lim_{c\to 0^+} {e^{-cn}}  = \lim_{c\to 0^+}\eta_c^s(Q_n^s)
    = 
    \lim_{c\to 0^+} \int Q_n^sf_c\d\nu_s= 
    \int \lim_{c\to 0^+} Q_n^sf_c \d\nu_s = \int Q_n^sf_0\d\nu_s.
    \end{align*}
    Let us end with the case $c\sl 0$, which mainly follows from two observations. The first is that 
    $$
    f_c\d \nu_s = 
    \left(
    1-\frac{e^{-c}-e^c}{e^{-c}+\sqrt s}
    \right)
    f_{-c}\d\nu_s,
    $$
    thus allowing us to harvest the moment equalities from the positive case. The second is 
    $$
    Q_n^s(e^c+\sqrt s+e^{-c}) = \frac{e^{-(n+1)c}+
    \sqrt s(e^{-{nc}}-e^{cn})
    -e^{(n+1)c}}{e^{-c}-e^c},\quad n\in \N.
    $$
    Which can easily be achieved by induction. With these observations in mind, we compute 
    \begin{align*}
        \eta_c^s(Q_n^s) &= \nu_s\big( 
        Q_n^s f_c
        \big) + \frac{e^{-c}-e^c}{e^{-c}+\sqrt s} Q_n^s\big(e^c+\sqrt s+e^{-c}\big)
        \\ &= 
        \left(
    1-\frac{e^{-c}-e^c}{e^{-c}+\sqrt s}
    \right) \eta_{-c}^s (Q_n^s)
        + \frac{e^{-(n+1)c}+
    \sqrt s(e^{-{nc}}-e^{cn})
    -e^{(n+1)c}}{e^{-c}+\sqrt s} \\
    &= \left(
    1-\frac{e^{-c}-e^c}{e^{-c}+\sqrt s}
    \right) e^{nc} + e^{-nc} - e^{nc}\frac{\sqrt s + e^c}{e^{-c}+\sqrt s}\\&= e^{-nc}.
    \end{align*}
\end{proof}

\section[\texorpdfstring{$H_N^{s+}$}{H\_N^{s+}}]
        {Defining the Brownian motion on \texorpdfstring{$H_N^{s+}$}{H\_N^{s+}}}\label{construction}

Defining a Brownian motion on the quantum reflection group $H_N^{s+}$ might seem odd, at least for two reasons. The first is that its classical counterpart is finite and therefore does not have a well-defined Brownian motion (note that $H_N^{s+}$ is nevertheless an infinite quantum group). The second is that one could consider Brownian motions as Gaussian processes restricted to the central algebra. In the case of $H_N^{s+}$, there is no non-trivial such process.

Before focusing our attention, let us briefly recall the situation for other typical quantum groups -- namely, the permutation, orthogonal, and unitary quantum groups $S_N^+$, $O_N^+$, and $U_N^+$ -- for which Brownian motions have already been defined.

Let us now describe the Brownian motion on $S_N^+ = H_N^{1+}$, with generators denoted by $(p_{ij})_{1 \leq i,j \leq N}$. It is a subgroup of $H_N^{s+}$ in the sense that there exists a surjective homomorphism intertwining the coproducts, $\pi: \mathcal{O}(H_N^{s+}) \to \mathcal{O}(S_N^+)$, $u_{ij} \mapsto p_{ij}$. Its irreducible characters, described by T. Banica in \cite{Ban99}, are given by:
$$
\chi_0 = 1, \quad \chi_1 = \sum_i p_{ii} - 1, \quad \text{and} \quad \chi_1 \chi_n = \chi_{n+1} + \chi_n + \chi_{n-1}, \quad n \geq 1.
$$
The earlier remark about $H_N^{s+}$ also applies to $S_N^+$: its classical counterpart is finite, and it has no non-trivial centralized Gaussian processes. However, as shown by U. Franz, A. Kula, and A. Skalski in \cite{FKS16}, any central generating functional on $S_N^+$ is of the form:
\begin{equation}
    \label{Brownian motion on S_N^+}
    P(\chi_1) \longmapsto -bP'(N) - \int_{-N}^N \frac{P(N)-P(x)}{N-x}\, \d\nu(x),
\end{equation}
for some $b \geq 0$ and positive measure $\nu$ on $[-N, N)$. Comparing this formula to the one proved by M. Liao for classical compact Lie groups in \cite{Lia04}, it was remarked in \cite{FTW21} that the process with $\nu = 0$ in the above decomposition plays an analogous role to the Laplace–Beltrami operator and was therefore called the \emph{Brownian motion}.

Following this approach, we can identify three main ways in which Brownian motions have been defined in the quantum setting:
\begin{enumerate}[label = \arabic*.]
    \item A Hunt-type formula for central generating functionals.
    \item Centralizing the classical Brownian motion onto the quantum algebra.
    \item Centralizing Gaussian generating functionals.
\end{enumerate}

As mentioned before, the first case works for $S_N^+$, while the other two yield nothing. For $O_N^+$ -- as presented in \cite[Subsec 
5.2]{FFS24} or \cite[App A]{Del24} for the first, \cite{FTW21} and \cite[Thm 10.2]{CFK14} for the second, and \cite[Prop 3.9]{BGJ20} for the third -- all three approaches coincide. As for $U_N^+$, it was stated in \cite{Del24} that the last two coincide (see \cite[Prop 3.2]{Del24} and \cite[Subsec 5.2]{FFS24} specifically for the second and third items, respectively). However, for $H_N^{s+}$, none of these techniques yield anything\footnote{For an overview of Gaussian generating functionals on bialgebras, see \cite[Sec. 1.5]{FS16}, and for their classification on the quantum groups we consider, see \cite{FFS24}. In particular, $H_N^{s+}$ has no non-zero Gaussian generating functionals, simply because such functionals vanish on projections, and $\O(H_N^{s+})$ is generated by projections.}. This can be summed up in the following tab. 

\begin{table}[h!]
    \centering
    \begin{tabular}{lccc}
        \toprule
        & Hunt-type formula & Centralized Gaussian functional & Classical Brownian motion \\
        \midrule
        $O_N^+$   & \checkmark & \checkmark & \checkmark \\
        $U_N^+$   & $\times$ & \checkmark & \checkmark \\
        $S_N^+$   & \checkmark & $\times$ & $\times$ \\
        $H_N^{s+}$ & $\times$ & $\times$ & $\times$ \\
        \bottomrule
    \end{tabular}
    \caption{Ways of defining Brownian motions across quantum groups.}
\end{table}

We will therefore make use of the following remark to define a proper Brownian motion on quantum reflection groups.

Note the following diagram of quantum subgroups:
\[
\begin{tikzcd}
    S_N^+ \simeq H_N^{1+} \arrow[d, "\leq"] \arrow[r, "\leq"] & H_N^{2+} \arrow[r, "\leq"] & O_N^+ \arrow[r, "\leq"] & U_N^+ \\
    H_N^{s+} \arrow[r, "\leq"] & H_N^{rs+} \arrow[r, "\leq"] & H_N^{\infty+} \arrow[ur, "\leq"]
\end{tikzcd}
\]
As shown above, the permutation quantum group $S_N^+$ is a quantum subgroup of all the other quantum groups introduced here. Moreover, one easily checks that pulling back its Brownian motion through the respective quotient maps and centralizing the functional yields Brownian motions on $O_N^+$ and $U_N^+$.

We will use this method of pulling back the Brownian motion from \( S_N^+ \) to \( H_N^{s+} \). Before doing so, we need to understand better the quotient map \( \pi: \mathcal{O}(H_N^{s+}) \to \mathcal{O}(S_N^+) \). Let us recall that \( M \) denotes the submonoid introduced in Theorem \ref{irr characters}, and restate the result established in \cite[Prop 2.1]{Lem15}.

\begin{prop}\label{Chebyshev}
    Let $w = a^{\ell_1}z^{r_1}\cdots z^{r_{k-1}}a^{\ell_k} \in M$, we have 
    $$
    \pi(\chi_w) = Q_{\underline \ell}(\chi_1),\quad \hbox{with } Q_{\underline \ell} = P_{\ell_1}(\sqrt X)
    \cdots
    P_{\ell_k}(\sqrt X),
    $$
    where the $P_n$'s are the Chebyshev polynomials of the second kind defined as follows 
    $$
    P_0 = 1,\quad P_1 = X \emph{\and} XP_n = P_{n+1}+P_{n-1},\quad n\in \N.
    $$
\end{prop}

\begin{rem}
    Note that \( Q_{\underline{\ell}} \) is a polynomial. Indeed, \( P_n \) is a polynomial containing only even powers when \( n \) is even. Moreover, following the induction relation satisfied by the Chebyshev polynomials, \( P_{\ell_1} \cdots P_{\ell_k} \) can be expressed as a linear combination of polynomials \( P_n \), where \( n \) shares the same parity as \( \ell_1 + \cdots + \ell_k \).
\end{rem}

\begin{rem}
    In the original proposition in \cite{Lem15}, everything is squared so as to avoid the square root in the polynomials. One easily adapts the proof to have this more general result. 
\end{rem}

We now have everything needed to define the Brownian motion on $H_N^{s+}$.

\begin{de}
    Let $L$ be the Brownian motion on $S_N^+$, i.e. the central generating functional corresponding to $b = 1$ and $\nu = 0$ in Equation \eqref{Brownian motion on S_N^+}. Then, we call \emph{Brownian motion} on $H_N^{s+}$, the generating functional $\Hat L = L\circ \pi \circ \E$ defined on $\O(H_N^{s+})$ where $\pi$ is the quotient map $\O(H_N^{s+})\to \O(S_N^+)$ and $\E$ the conditional expectation $\O(H_N^{s+})\to \O(H_N^{s+})_0$. In particular, we have 
    $$
    \Hat L (\chi_{a^{\ell_1}z^{r_1}\cdots z^{r_{k-1}}a^{\ell_k}}) = -Q_{\underline \ell}'(N), \quad 
    a^{\ell_1}z^{r_1}\cdots z^{r_{k-1}}a^{\ell_k}\in M.
    $$
    We will also call \emph{Brownian motion} the associated Lévy process $(\vphi_t)_{t\ge 0}$ 
    $$
    \vphi_t(\chi_{a^{\ell_1}z^{r_1}\cdots z^{r_{k-1}}a^{\ell_k}}) = Q_{\underline \ell}(N) \exp\left(
    -t \frac{Q'_{\underline \ell}(N)}{Q_{\underline \ell}(N)}
    \right),\quad a^{\ell_1}z^{r_1}\cdots z^{r_{k-1}}a^{\ell_k}\in M.
    $$
\end{de}

\section{Computing the limit profile}\label{profile}

\subsection{Finding a commutative subalgebra}

In this subsection, our goal is to identify the information that matters asymptotically for the Brownian motion.

\begin{lem}\label{type}
    Let $w = a^{\ell_1}z^{r_1}\cdots z^{r_{k-1}}a^{\ell_k}\in M$, call $n := (\ell_1+\cdots+\ell_k)/2\in \N$ its \emph{type}, then setting $t_N = N\ln N + cN$ for some $c\in \R$, we have 
    $$
    \vphi_{t_N}(\chi_w) \underset{N\to\infty}{\longto} e^{-cn}.
    $$
\end{lem}

\begin{proof}
Simply remark that $ Q_{\underline{\ell}} $ is a unital polynomial of degree $n$. From this, it follows that
\begin{align*}
    \vphi_{t_N} (\chi_w) &= Q_{\underline \ell}(N) \exp\left(
    -t \frac{Q'_{\underline \ell}(N)}{Q_{\underline \ell}(N)}
    \right) \\ &= N^n (1+o(1))
    \exp\left(
    -N(\ln N + c) \frac{n+O(N^{-1})}{N}
    \right) \\ &= 
    (1+o(1))e^{-n c+o(1)} \\ &\underset{N\to \infty}{\longto} e^{-cn}.
\end{align*}
\end{proof}

The key information happens to be what we called \emph{type}. 
Let \( M_n \) be the set of irreducible characters of type \( n \) for any \( n \in \mathbb{N} \). Specifically, \( M_0 \) contains only the empty set and for any $n\in \N^*$, we have   
\[
  M_n = \{ az^{r_1}a^2 \cdots a^2z^{r_n}a : 0 \leq r_i \leq s-1 \}.
  \]  
Note that \( |M_n| = s^n \).  

The final goal of this subsection is to identify a subalgebra that simplifies the analysis while retaining the essential information about character types. The following proposition summarizes the required results. 

\begin{prop}\label{cond exp}
    Let $x_n$ be the normalised sum of all elements of type $n$. In other words, 
    $$
    x_n = \frac{1}{\sqrt {s^n}}\sum_{w\in M_n}\chi_w.
    $$
    Let $\O(H_N^{s+})_{00}$ be the subalgebra generated by the selfadjoint element $x_1$. Then, the following assertions hold:
    \begin{enumerate}[label = \emph{(\roman*)}]
        \item For any $n\ge 1$, we have $x_1x_n = x_{n+1}+\sqrt s x_n + x_{n-1}$. In other words, we have $x_n = Q_n^s(x_1)$ for any $n\in \N$.
        \item There is a conditional expectation $\F:\O(H_N^{s+})_0\to \O(H_N^{s+})_{00}$ that leaves the Haar state invariant and is defined by 
        $$
        \F[\chi_w] = \frac{x_n}{\sqrt {s^n}},\quad 
        \hbox{whenever }w \hbox{ is of type } n.
        $$
    \end{enumerate}
\end{prop}

\begin{proof}
    (i) Let $n \ge 1$. We simply compute:
    \begin{align*}
        x_1x_n &= 
        \frac{1}{\sqrt {s^{n+1}}}
        \sum_{r_0,\cdots,r_n} 
        \chi_{az^{r_0}a} \cdot \chi_{az^{r_1}a^2\cdots a^2z^{r_{n}}a} \\&= \frac{1}{\sqrt {s^{n+1}}}
        \sum_{r_0,\cdots,r_n} \chi_{az^{r_0}a^2\cdots a^2z^{r_{n}}a} + \chi_{az^{r_0}}\cdot \chi_{z^{r_1}a^2\cdots a^2z^{r_{n}}a} \\ &=
        x_{n+1} + \frac{1}{\sqrt {s^{n+1}}}\sum_{r_1,\cdots,r_n} \left(
        \chi_{az^{0}a^2z^{r_2}\cdots a^2z^{r_{n}}a} + \chi_{az^{r_2}a^2\cdots a^2z^{r_{n}}a}
        \sum_{r_0\ne s-r_1} \chi_{az^{r_0+r_1}a^2\cdots a^2z^{r_{n}}a} 
        \right) \\ &= x_{n+1} + 
        \frac{1}{\sqrt {s^{n+1}}}
        \sum_{r_1,\cdots,r_n} \left( \chi_{az^{r_2}a^2\cdots a^2z^{r_{n}}a}+
        \sum_{r_0} \chi_{az^{r_0+r_1}a^2\cdots a^2z^{r_{n}}a} 
        \right)\\ &= x_{n+1} + x_{n-1}+ \sqrt sx_n.
    \end{align*}
    (ii) Let $\mathrm L^2(H_N^{s+})_0$ and $\mathrm L^2(H_N^{s+})_{00}$ be the $\mathrm L^2$-spaces of $\O(H_N^{s+})_0$ and $\O(H_N^{s+})_{00}$ respectively (with $\langle x,y\rangle = h(xy^*)$ as the inner product) and $\F'$ be the orthogonal projection $\mathrm L^2(H_N^{s+})_0\to \mathrm L^2(H_N^{s+})_{00}$. Let us show that $\F'$ satisfies the relation stated in the theorem, it suffices to see that 
$$
\langle \chi_w, x_n\rangle = \left\{
\begin{array}{cl}
    1/\sqrt{s^n} &\hbox{ if }w \hbox{ is of type }n  \\
    0 &\hbox{ else,} 
\end{array}
\right. \qquad w\in M, \,\, n\in \N.
$$
Which simply follows from the fact that the irreducible characters form an orthonormal family for the inner product. Since the Haar state is faithful, the algebra $\O(H_N^{s+})$ embeds into the von Neumann algebra $\mathrm L^\infty(H_N^{s+})$ so that it follows from \cite[Thm 9.1.2]{AP} that restricting $\F = \F'\circ\E$ gives a conditional expectation.
\end{proof}

\begin{cor}\label{Haar state subalgebra of H_N^s+}
    The spectral measure of 
    $$
    x_1 = \frac{1}{\sqrt s} \sum_{r=0}^{s-1}\chi_{az^ra},
    $$
    under the Haar state is given by the free Meixner law $\nu_s = \Meix^+(\sqrt s,1)$.
\end{cor}

\begin{proof}
    This simply follows from the first item of Proposition \ref{cond exp}.
\end{proof}

\subsection{The limit profile}

We have identified a commutative subalgebra such that the process is “asymptotically” invariant under the orthogonal projection. We outline that the process is however, not truly $\F$-invariant: consider, for instance,
\[
\varphi_t(\chi_{a^2}) = (N-1)e^{-t/(N-1)} 
\ne Ne^{-t/N} = \varphi_t(\chi_{aza}).
\]
We will therefore make use of the invariant process $\widetilde \vphi = (\vphi_t\circ \F)_{t\ge0}$. 

First, we investigate absolute continuity. 

\begin{lem}\label{abs cont}
Let $c\in \R$ and set $t_N = N\ln(\sqrt sN)+cN$, then 
\begin{itemize}
    \item if $c\sg 0$, then $\vphi_{t_N}/\widetilde \vphi_{t_N}$ is absolutely continuous w.r.t. the Haar state, provided $N$ is large enough;
    \item if $c\sl 0$, then $\vphi_{t_N}/\widetilde \vphi_{t_N}$ is not absolutely continuous w.r.t. the Haar state, provided $N$ is large enough. 
\end{itemize}
\end{lem}

\begin{proof}
    Set $c\sg 0$.
    Recall that the $P_n$'s denote the Chebyshev polynomials of second kind as defined in Proposition \ref{Chebyshev}. Using the estimate from \cite[Lem 1.7]{FHLUZ17} stating that 
    $$
    \frac{P_n'(x)}{P_n(x)} \ge \frac{n}{x},\quad n\in \N,\,\, x\ge 2.
    $$
    and the fact that $P_n(x)\le x^n$ for any $n\in \N$ and $x\sg0$, we have, for any $w = a^{\ell_1}z^{r_1}\ldots z^{r_{k-1}}a^{\ell_k}\in M_s$ of type $n\in \N$:
    \begin{align*}
        \vphi_{t_N}(\chi_w)^2 = Q_{\underline \ell}(\sqrt N)^2 e^{-2t_NQ'_{\underline \ell}(\sqrt N)/Q_{\underline \ell}(\sqrt N)} \le N^{2n}\exp\left(
        -2t_Nn/N
        \right) = \frac{e^{-2nc}}{s^n}.
    \end{align*}
    It only remains to compute 
    \begin{align*}
        \Vert \vphi_{t_N}\Vert_{2}^2 \le 
    \sum_{n\in \N} \sum_{w\in M_n} \vphi_{t_N}(\chi_w)^2 \le \sum_{n\in \N}e^{-2nc}\sl \infty
    \end{align*}
    and using the second item of Proposition \ref{cond exp}
    \begin{align*}
        \Vert \widetilde \vphi_{t_N}\Vert_{2}^2 &\le 
    \sum_{n\in \N} \sum_{w\in M_n}\widetilde  \vphi_{t_N}(\chi_w)^2  = \sum_{n\in \N} \sum_{w\in M_n}  \vphi_{t_N}\circ \F(\chi_w)^2 \le \sum_{n\in \N}e^{-2nc}\sl \infty.
    \end{align*}
    For the case $c\sl 0$, we restrict to the commutative subalgebra \( \mathcal{O}(H_N^{s+})_{00} \), we may view the states \( \varphi_{t_N} \) (or \( \widetilde{\varphi}_{t_N} \), they coincide when restricted to this subalgebra) and the Haar state as probability measures \( \mu_{t_N} \) and \( \nu_s = \Meix^+(\sqrt{s}, 1) \) via the isomorphism $\mathcal{O}(H_N^{s+})_{00} \to \mathbb{C}[X], x \mapsto X$. Lemma \ref{type} then implies that
\[
\mu_{t_N}(P_m) \underset{N \to \infty}{\longrightarrow} e^{-cm}, \quad m \in \mathbb{N}.
\]
In other words, \( \mu_{t_N} \) converges in moments to $\eta_c^s$. Since \( c \sl 0 \), it follows from Lemma \ref{decompo of eta_c^s} that the support of \( \nu_s \) is contained in the symmetric interval
\[
[-2 - \sqrt{s},\, 2 + \sqrt{s}],
\]
and that $\eta_c^s$ has an atom located outside this interval, at \( e^c + e^{-c} + \sqrt{s} \). This implies that, for \( N \) large enough, \( \mu_{t_N} \) must also allocate some mass outside this symmetric support, and hence absolute continuity does not hold.
\end{proof}

Finally, we compute the cutoff profile. 

\begin{thm}\label{cutoff profile Brownian of H_N^s+}
    The Brownian motion on $H_N^{s+}$ has cutoff at time $t_N = N\ln N$. More precisely, it exhibits the following cutoff profile where we write $d_N(t) := d_{\mathrm{TV}}(\vphi_t,h)$:
    $$
    d_N\left(
    N\ln \left(
    \sqrt sN
    \right)+cN
    \right) \underset{N\to\infty}{\longto}
    d_{\mathrm{TV}}
    \left(
    \eta_c^s,\nu_{s}
    \right)
    ,\quad c \sg 0,
    $$
    and 
    $$
    \underset{N\to\infty}{\lim\inf}\, d_N\left(
    N\ln \left(
    \sqrt sN
    \right)+cN
    \right) \ge  d_{\mathrm TV} (\eta_c^s,\nu_s),\quad c\in \R.
    $$
    where $\nu_{s}$ and $\eta_c^s$ are defined as in the end of Section \ref{preliminaries}. 
\end{thm}

\begin{proof}
    We begin by proving the result for the $\F$-invariant process $\widetilde{\vphi}_t := \vphi_t \circ \F$. By invariance, its study can be restricted to the commutative algebra $\O(H_N^{s+})_{00} = \C[x_1]$. From the first item of Proposition \ref{cond exp}, it follows that the $x_n$'s form an orthonormal basis for this space. In particular, the states $\widetilde{\vphi}_t$ and $h$ can be viewed as probability measures when composed with the isomorphism $\O(H_N^{s+})_{00} \to \C[X],x_1 \mapsto X$. In this context $h$ corresponds to $\nu_{s}$ and $\widetilde{\vphi}_{N\ln (\sqrt sN)+cN}$ converges to $\eta_c^s$ in moments for all $c\in \R$ (see Lemma \ref{type} and Corollary \ref{Haar state subalgebra of H_N^s+}).

    Let $c \sg 0$ and set $t_N : =\ln(\sqrt sN)+cN$. By Lemma \ref{abs cont}, $\widetilde \vphi_{t_N}$ is absolutely continuous w.r.t $h$, with density given by 
    \[
    f_{t_N} := \sum_{n \in \N} \widetilde{\vphi}_{t_N}(x_n)Q^s_n.
    \]
    Moreover, it follows from the proof of Lemma \ref{abs cont} that this density is uniformly dominated, independently of $N$, by 
    \[
    \sum_{n \in \N} e^{-cn}Q^s_n \in \mathrm{L}^2(\nu_s).
    \]
    Additionally, for any $n \in \N$, Lemma \ref{type} yields the earlier mentioned convergence in moments
    \[
    \widetilde{\vphi}_{t_N}(x_n) = \frac{1}{\sqrt{s^n}} \sum_{w \in M_n} \vphi_{t_N}(\chi_w) \underset{N \to \infty}{\longto} e^{-nc}.
    \]
    By dominated convergence, the summation and the limit in $N$ can be interchanged, leading to the desired $\mathrm{L}^1$-convergence:
    \begin{align*}
        d_{\mathrm{TV}}(\widetilde{\vphi}_{t_N}, h) &= \frac{1}{2} \|f_{t_N} - 1\|_1 = \frac{1}{2} \left\| \sum_{n \geq 1} \vphi_{t_N}(x_n) Q_n^s \right\|_{\nu_s,1} \\ &\underset{N \to \infty}{\longto} \frac{1}{2} \left\| \sum_{n \geq 1} e^{-nc} Q_n^s \right\|_{\nu_s,1} = d_{\mathrm{TV}}(\eta_c^s, \nu_s).
    \end{align*}

    Now let $c \sl 0$. As in the previous case, for any $n \in \N$, we have
    \[
    \vphi_{t_N}(Q_n^s) \leq e^{-nc} = \eta_c^s(Q_n^s).
    \]
    This implies that 
    \[
    \vphi_{t_N}(X^{2n}) \leq \eta_c^s(X^{2n}), \quad n \in \N,
    \]
    since $X^{2n}$ can be expressed as a linear combination of $Q_n^s$'s with positive coefficients. As $\eta_c^s$ has compact support, $\widetilde{\vphi}_{t_N}$ must also have compact support (see, for instance, \cite[Lect 2]{Sch20}). Given that moment convergence with compact support implies weak convergence, we have
    \[
    \widetilde{\vphi}_{t_N}(B) \xrightarrow[N \to \infty]{} \eta_c^s(B),
    \]
    for any continuity set $B$ of $\eta_c^s$, i.e., a Borel set $B \subset \R$ such that $\eta_c^s(\partial B) = 0$.
    To maximize the total variation distance between $\eta_c^s$ and $\nu_s$, we choose
    \[
    B_c = \{x \in (\sqrt{s} - 2, \sqrt{s} + 2) : f_c^s(x) \sl 1\} \cup V_s,
    \]
    where $V_s$ is a sufficiently small open neighborhood of $-1/\sqrt{s}$ (with $V_1 = \emptyset$ for $s = 1$), and $f_c^s$ is the density of the absolutely continuous part of $\eta_c^s$ w.r.t $\nu_s$. Weak convergence then implies:
    \begin{align*}
        d_{\mathrm{TV}}(\widetilde{\vphi}_{t_N}, h) &\geq \nu_s(B_c) - \vphi_{t_N}(B_c) \underset{N \to \infty}{\longto} \nu_s(B_c) - \eta_c^s(B_c) = d_{\mathrm{TV}}(\eta_c^s, \nu_s).
    \end{align*}

    The result is thus proven for the process centralized onto the commutative algebra. For the original process, the case $c \sl 0$ follows from the fact that the conditional expectation has norm 1:
    \begin{align*}
        \|\widetilde{\vphi}_{t} - h\|_1 &= \|(\vphi_{t} - h) \circ \F\|_1 \leq \|(\vphi_{t} - h)\|_1 \cdot \|\F\| \leq \|(\vphi_{t} - h)\|_1.
    \end{align*}

    For $c \sg 0$, (still writing $t_N = N\ln(\sqrt sN)+cN$) Lemma \ref{abs cont} shows that both $\vphi_{t_N}$ and $\widetilde{\vphi}_{t_N}$ are absolutely continuous with respect to $\nu_s$. Furthermore, their densities are uniformly bounded in $N$ by an integrable function, ensuring $\mathrm{L}^2$- and (therefore) $\mathrm{L}^1$-convergence by domination. Since they share the same moment convergence, we conclude:
    \begin{align*}
        \|\vphi_{t_N} - \widetilde{\vphi}_{t_N}\|_1^2 &\leq \left\| \sum_w (\vphi_{t_N} - \widetilde{\vphi}_{t_N})(\chi_w) \right\|_2^2 \leq \sum_w |\vphi_{t_N}(\chi_w) - \widetilde{\vphi}_{t_N}(\chi_w)|^2 \underset{N \to \infty}{\longto} 0.
    \end{align*}
\end{proof}

\begin{rem}\label{retranslated profiles}
    Note that we have intentionally translated the profile by \(\ln\sqrt{s}\) (using \(t_N = N\ln(\sqrt{s}N) + cN\) instead of \(N\ln N + cN\)) to simplify computations. This adjustment ensures that absolute continuity is lost precisely when \(c < 0\).  

    We highlight this remark as our next goal is to understand the influence of the parameter \(s\) in shaping the limit profile. Let \(f_s\) denote the retranslated limit profile. In other words:  
    \[
    f_s : c \longmapsto d_{\mathrm{TV}}(\eta^s_{c-\ln\sqrt s}, h).
    \]
    The total variation distance can be determined by integrating the absolutely continuous part and manually accounting for any remaining atoms.  

    Furthermore, as \(c \to \infty\), we observe the weak convergence:  
    \[
    f_s \underset{c\to\infty}{\longto} f_\infty = \left\{  
    c \mapsto \frac{e^{-c}}{1 + e^{-c}}  
    \right\}.  
    \]
    Below is a plot of the functions \(f_s\) for \(s = 1, 2, 4, 9\) and \(\infty\):  
    \begin{figure}[h]  
    \centering  
    \includegraphics[width=0.6\textwidth]{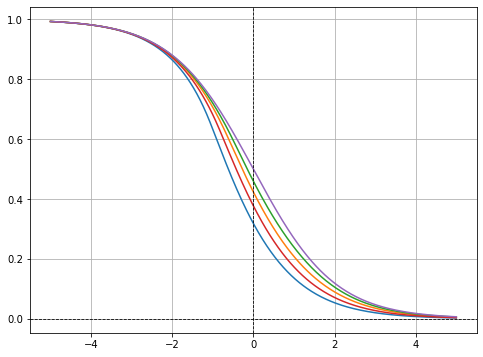}  
    \caption{Graph comparing the functions \(f_1\), \(f_2\), \(f_4\), \(f_9\), and \(f_\infty\) (shown in blue, red, orange, green and purple, respectively).}  
    \label{fig:myimage}  
    \end{figure}  
\end{rem}

\newpage

\begin{prop}\label{different s different profile}
    Every $s\in \N^*$ defines a unique cutoff profile (unique up to affine transformation) through the formula
    $$
    c\longmapsto d_{\mathrm{TV}}\left( 
    \nu_s , \eta_c^s
    \right).
    $$
\end{prop}

\begin{proof}
    Consider the retranslated profiles defined in Remark \ref{retranslated profiles}: 
    $$
    f_s = \left\{ 
    c\longmapsto d_{\mathrm{TV}}\left( 
    \nu_c^s,\nu_s
    \right)
    \right\},\quad s\ge 1.
    $$
    where we have written $\nu_c^s = \eta^s_{c-\ln\sqrt s}$. Recall the decomposition from Lemma \ref{decompo of eta_c^s}:
    $$
    \nu_c^s = g_c^s \d\nu_s + \mathbf 1_{c\sl \ln \sqrt s}\frac{e^{-c}-e^c/s}{e^{-c}+1}\delta_{\sqrt s(e^{-c}+1)+ e^c/{\sqrt s}},
    $$
    with 
    $$
    g_c^s(x) = 
    \frac{\sqrt s + e^c/\sqrt s}{\sqrt s + e^{-c}\sqrt s+e^c/\sqrt s - x}
    ,\quad 
    x \in [\sqrt s-2,\sqrt s+2].
    $$
    Note that $g_c^s \downarrow 0$ as $c\to -\infty$, this ensures that $\vert g_c^s\vert \le 1$, in an interval of the form $(-\infty, c_s]$. This further implies that, in that case, the total variation distance between $\nu_c$ and $\nu_c^s$ is carried by the singular part. In other words, for every 
    $c\in (-\infty,c_s]$, we have 
    \begin{align*}
        f_s(c) &= \nu_c^s\big(\{\sqrt s(e^{-c}+1)+e^c/\sqrt s\}\big) - \nu_s\big(\{\sqrt s(e^{-c}+1)+e^c/\sqrt s\}\big) \\ &= 
        \frac{e^{-c}-e^c/s}{e^{-c}+1}.
    \end{align*}
    Now fix $r, s \in \N^*$ and assume $f_r$ and $f_s$ are equivalent modulo affine transformation so that there exist $a\sg 0$ and $b\in \R$, such that 
    $$
    f_r(ac+b) = f_s(c),\quad c\in \R.
    $$
    From what we previously established, this implies that for every $c$ in an interval of the form $(-\infty, c_0]$, we must have 
    $$
    \frac{e^{-(ac+b)}-e^{ac+b}/r}{e^{-(ac+b)}+1} = \frac{e^{-c}-e^c/s}{e^{-c}+1}.
    $$
    Setting $y := e^{-c}$, this can be rewritten as 
    $$
    \left(y^ae^{-b}-\frac{y^{-a}e^b}{r}\right)
    (y+1) = \left(y-\frac{y^{-1}}{s}\right)(y^ae^{-b}+1),
    $$
    Since this holds for infinitely many $y$, the exponents on the left and right hand side must match. More precisely, we have 
    $$
    \{1+a,1-a,a,-a\}
    = 
    \{1+a,a-1,1,-1\}.
    $$
    In particular, $1$ is in $\{1+a,1-a,a,-a\}$ and since $a>0$, the only possibility is that $a=1$. Knowing this, we rearrange the equality to get
    $$
    y\big( 
    e^{-b} - 1
    \big)  + y^{-1} \left( 
    \frac{1}{s}-\frac{e^b}{r}
    \right) + \left(
    \frac{e^{-b}}{s}-\frac{e^b}{r}
    \right) = 0,
    $$
    easily leading to $b=0$ and $s=r$. 
\end{proof}

\begin{rem}
    One can easily make an incorrect conjecture -- when inspecting Figure \ref{fig:myimage} -- that the $f_s$'s (the retranslated profiles) initially appear to be the same (specifically in the region where all the weight is carried by the singular part) before eventually leading to distinct profiles. It follows from the proof of Proposition \ref{different s different profile} that, in the region where the total variation distance is determined only by the singular part, their expression is given by 
    $$
    f_s: c\longmapsto \frac{e^{-c}-e^c/s}{e^{-c}+1},
    $$
    the dependency in $s$ is very small for $c$ in a neighborhood of $-\infty$ but it still holds. 
\end{rem}

\section[\texorpdfstring{$H_N^{\infty+}$}{H\_N^{∞+}}]
        {The case of \texorpdfstring{$H_N^{\infty+}$}{H\_N^{∞+}}}\label{654665768}

The irreducible characters of $H_N^{\infty+}$ are similar to the finite case, let us restate it for clarity. 
The nontrivial characters are labeled by the elements of the submonoid 
$$
M_\infty = \langle 
az^ra : r\in \Z
\rangle \subset \langle a,z,z^{-1}\rangle,
$$
such that 
$$
\chi_{a^2} = \sum_{i=1}^N p_{ii}-1 \and \chi_{az^ra}
= \sum_{i=1}^N h_{ii}^r
,\quad r\in \Z\setminus \{0\},
$$
and 
$$
    \chi_{vaz^i}\cdot \chi_{z^jaw} = 
    \left\{
    \begin{array}{ll}
    \chi_{vaz^{i+j}aw} & \hbox{if }i+j\ne0 \\
    \chi_{va^2w}+ \chi_v\cdot \chi_w &\hbox{else}. 
\end{array}\right.
$$

In the infinite case, there are infinitely many elements of each type. Namely, it does not make sense to consider an element that would be the normalised sum of all elements of a certain type. We therefore make a few conjectures. 

\begin{conj}
Let $\vphi$ denote the standard Brownian motion (as a Lévy process) on $H_N^{\infty+}$. 
\begin{enumerate}[label = \emph{(\roman*)}]
    \item For every $t\ge0$, the state $\vphi_t$ is not absolutely continuous w.r.t. the Haar state. 
    \item Setting $d_N(t) := d_{\mathrm{TV}}(\vphi_t,h)$, we have 
    $$
    \underset{N\to\infty}{\lim\inf}\, d_N\left(
    N\ln N + cN
    \right) \ge  \frac{e^{-c}}{1+e^{-c}} = f_\infty(c),\quad c\in \R.
    $$
\end{enumerate}
\end{conj}

\begin{proof}[Idea of proof]
    An idea that could work to prove this conjecture is that the Brownian motion on $H_N^{\infty +}$ restricted to the subalgebra generated by an element of the form 
    $$
    x = \sum_{r = 1}^s \big(\chi_{az^ra}
    +\chi_{az^{-r}a}
    \big)
    ,
    $$
    should behave in a similar fashion as the Brownian motion on $H_N^{2s+}$. In particular, absolute continuity should be lost for $t\le N\ln N - N\ln \sqrt {2s}$, and one should have a lower bound of the form 
    $$
    \underset{N\to\infty}{\lim\inf}\, d_N\left(
    N\ln N
    +cN
    \right) \ge  \widetilde f_{2s}(c),\quad c\in \R.
    $$
    where $\widetilde f_{2s}$ should be close to $f_{2s}$. Both of these results holding for any $s$, this would imply both a loss of absolute continuity everywhere and a lower bound of the form 
    $$
    \underset{N\to\infty}{\lim\inf}\, d_N\left(
    N\ln N
    +cN
    \right) \ge  \sup_{s}\widetilde f_{2s}(c) = f_\infty(c),\quad c\in  \R.
    $$
\end{proof}

\begin{rem}
    Recall that, given a compact quantum group $\G$, a central state $\vphi$ has an $\mathrm L^2$-density with respect to the Haar state if and only if 
$$
\sum_{\alpha \in \Irr(\G)} \vert 
\vphi(\chi_\alpha)
\vert ^2 < \infty.
$$
It is easy to see that $\vphi_t$ does not have an $\mathrm L^2$-density: for example, all values $\vphi_t(\chi_{az^ra})$ for $r \in \Z^*$ are equal and nonzero, so that
$$
\sum_{r \in \Z^*} \vert \vphi_t(\chi_{az^ra})\vert^2 = \infty.
$$
\end{rem}



\begin{thebibliography}{10}

\bibitem{Ald83}
D.~Aldous.
\newblock Random walks on finite groups and rapidly mixing {M}arkov chains.
\newblock {\em Séminaire de Probabilités. Proceedings. Springer Berlin Heidelberg}, S(XVII), 1983.

\bibitem{AP}
C.~Anantharaman and S.~Popa.
\newblock An introduction to {II}$_1$ factors.
\newblock {\em Preprint}, 2010.

\bibitem{Ban99}
T.~Banica.
\newblock Symmetries of a generic coaction.
\newblock {\em Math. Ann.}, 314(4):763--780, 1999.

\bibitem{BV09}
T.~Banica and R.~Vergnioux.
\newblock Fusion rules for quantum reflection groups.
\newblock {\em J. Noncommut. Geom.}, 3(3):327--359, 2009.

\bibitem{BD92}
D.~Bayer and P.~Diaconis.
\newblock Trailing the dovetail shuffle to its lair.
\newblock {\em Ann. Appl. Probab.}, 2:294–313, 1992.

\bibitem{BLPS18}
N.~Berestycki, E.~Lubetzky, Y.~Peres, and A~Sly.
\newblock Random walks on the random graph.
\newblock {\em Ann. Probab.}, 46(1):456--490, 2018.

\bibitem{Bla06}
B.~Blackadar.
\newblock Operator algebras, encyclopædia of math. sciences.
\newblock {\em Springer}, 122:165–210, 2006.

\bibitem{BB06}
M.~Bożejko and W.~Bryc.
\newblock On a class of free {L}évy laws related to a regression problem.
\newblock {\em J. Funct. Anal.}, 236(1):59--77, 2006.

\bibitem{BGJ20}
M.~Brannan, L.~Gao, and M.~Junge.
\newblock Complete logarithmic {S}obolev inequalities via {R}icci curvature bounded below {II}.
\newblock {\em J. Topol. Anal.}, 2020.

\bibitem{BR17}
M.~Brannan and Z-J. Ruan.
\newblock $\mathrm {L}^p$-representations of discrete quantum groups.
\newblock {\em J. Reine Angew. Math.}, 732:165–210, 2017.

\bibitem{BN22}
A.~Bufetov and P.~Nejjar.
\newblock Cutoff profile of {ASEP} on a segment.
\newblock {\em Probab. Theory Related Fields}, 183(1):229--253, 2022.

\bibitem{CFK14}
F.~Cipriani, U.~Franz, and A.~Kula.
\newblock Symmetries of {L}évy processes, their {M}arkov semigroups and potential theory on compact quantum groups.
\newblock {\em J. Funct. Anal.}, 266(5):2789--2844, 2014.

\bibitem{Del24}
J.~Delhaye.
\newblock Brownian motion on the unitary quantum group: construction and cutoff.
\newblock {\em preprint}, 2024.
\newblock \href{https://arxiv.org/pdf/2409.06552}{arXiv:2409.06552}.

\bibitem{Dia88}
P.~Diaconis.
\newblock Group representations in probability and statistics.
\newblock {\em Lecture Notes-Monograph Series, Institute of Math. Statistics}, 11, 1988.

\bibitem{DGM90}
P.~Diaconis, R.L. Graham, and J.A. Morrison.
\newblock Asymptotic analysis of a random walk on a hypercube with many dimensions.
\newblock {\em Random Structures Algorithms}, 1.1:51–72, 1990.

\bibitem{DS81}
P.~Diaconis and M.~Shahshahani.
\newblock Generating a random permutation with random transpositions.
\newblock {\em Probab. Theory Related Fields}, 57(2):159--179, 1981.

\bibitem{FFS24}
U.~Franz, A.~Freslon, and A.~Skalski.
\newblock {G}aussian generating functionals on easy quantum groups.
\newblock {\em J. Theoret. Probab.}, 2024.

\bibitem{FHLUZ17}
U.~Franz, G.~Hong, F.~Lemeux, M.~Ullrich, and H.~Zhang.
\newblock Hypercontractivity of heat semigroups on free quantum groups.
\newblock {\em J. Operator Theory}, 77(1):61--76, 2017.

\bibitem{FKS16}
U.~Franz, A.~Kula, and A.~Skalski.
\newblock L{\'e}vy processes on quantum permutation groups.
\newblock In {\em Noncommutative Analysis, Operator Theory and Applications}, volume 252 of {\em Oper. Theory Adv. Appl.}, pages 891--921. Birkhäuser, 2016.

\bibitem{FS16}
U.~Franz and A.~Skalski.
\newblock Noncommutative mathematics for quantum systems.
\newblock {\em Cambridge University Press}, 2016.

\bibitem{Fre18}
A.~Freslon.
\newblock Quantum reflections, random walks and cut-off.
\newblock {\em Internat. J. Math.}, 29(14), 2018.

\bibitem{Fre19}
A.~Freslon.
\newblock Cut-off phenomenon for random walks on free orthogonal quantum groups.
\newblock {\em Probab. Theory Related Fields}, 174(3–4):731--–760, 2019.

\bibitem{FTW21}
A.~Freslon, L.~Teyssier, and S.~Wang.
\newblock Cutoff profiles for quantum {L}évy processes and quantum random transpositions.
\newblock {\em Probab. Theory Related Fields}, 183:1285--1327, 2022.

\bibitem{Lac16}
H.~Lacoin.
\newblock Mixing time and cutoff for the adjacent transposition shuffle and the simple exclusion.
\newblock {\em Ann. Probab.}, 44(2):1426–1487, 2016.

\bibitem{Lem15}
F.~Lemeux.
\newblock Haagerup approximation property for quantum reflection groups.
\newblock {\em Proceedings of the American Math. Society}, 143(5):2017--2031, 2015.

\bibitem{LLP10}
D.~Levin, M.~Luczak, and Y.~Peres.
\newblock Glauber dynamics for the mean-field {I}sing model: cut-off, critical power law, and metastability.
\newblock {\em Probab. Theory Related Fields}, 146:223--265, 2009.

\bibitem{Lia04}
M.~Liao.
\newblock Lévy processes and {F}ourier analysis on compact {L}ie groups.
\newblock {\em Ann. Probab.}, pages 1553--1573, 2004.

\bibitem{LS10}
E.~Lubetzky and A.~Sly.
\newblock Cutoff phenomena for random walks on random regular graphs.
\newblock {\em Duke Math. J.}, 153(3):475--510, 2010.

\bibitem{Mcc17}
J.P. McCarthy.
\newblock Random walks on finite quantum groups : Diaconis-{S}hahshahani theory for quantum groups.
\newblock {\em PhD Thesis}, 2017.
\newblock \href{https://arxiv.org/pdf/1709.09357}{arXiv:1709.09357}.

\bibitem{Mcc19}
J.P. McCarthy.
\newblock Diaconis–{S}hahshahani upper bound lemma for finite quantum groups.
\newblock {\em J. Fourier Anal. Appl.}, 25(5):2463--2491, 2019.

\bibitem{Mel14}
P-L. Méliot.
\newblock The cut-off phenomenon for {B}rownian motions on compact symmetric spaces.
\newblock {\em Potential Analysis}, 40(4):427--509, 2014.

\bibitem{NT13}
S.~Neshveyev and L.~Tuset.
\newblock Compact quantum groups and their representation categories.
\newblock {\em Société Mathématique de France}, 20:165–210, 2013.

\bibitem{NT22}
E.~Nestoridi and S.~Olesker-Taylor.
\newblock Limit profiles for {M}arkov chains.
\newblock {\em Probab. Theory Related Fields}, 182:157--188, 2022.

\bibitem{NS06}
A.~Nica and R.~Speicher.
\newblock {\em Lectures on the Combinatorics of Free Probability}, volume 335 of {\em London Math. Society Lecture Note Series}.
\newblock Cambridge University Press, 2006.

\bibitem{OTT25}
S.~Olesker-Taylor, L.~Teyssier, and P.Thévenin.
\newblock Sharp character bounds and cutoff profiles for symmetric groups.
\newblock {\em preprint}, 2025.
\newblock , \href{https://arxiv.org/pdf/2503.12735}{arXiv:2503.12735}.

\bibitem{SY01}
N.~Saitoh and H.~Yoshida.
\newblock The infinite divisibility and orthogonal polynomials with a constant recursion formula in free probability theory.
\newblock {\em Probab. Math. Statist.}, 21(1):159--170, 2001.

\bibitem{Sch20}
K.~Schmüdgen.
\newblock Ten lectures on the moment problem.
\newblock {\em Lecture Notes}, 2020.
\newblock \href{https://arxiv.org/pdf/2008.12698}{arXiv:2008.12698}.

\bibitem{Tey19}
L.~Teyssier.
\newblock Limit profile for random transpositions.
\newblock {\em Ann. Probab.}, 48(5):2323–2343, 2019.

\bibitem{Wan95}
S.~Wang.
\newblock Free products of compact quantum groups.
\newblock {\em Comm. Math. Phys.}, 167(3):671--692, 1995.

\bibitem{Wor98}
S.L. Woronowicz.
\newblock Compact quantum groups, symétries quantiques.
\newblock {\em Les Houches}, pages 845--884, 1998.

\end{thebibliography}
\end{document}